\documentclass[a4paper]{article}
\usepackage{mathrsfs, amsmath, amsxtra, amssymb, latexsym, amscd, amsthm}
\newtheorem{thm}{Theorem}[section]

\newtheorem{lem}{Lemma}[section]
\newtheorem{prop}{Proposition}[section]
\theoremstyle{definition}
\newtheorem{defn}{Definition}[section]
\newtheorem{exam}{Example}[section]

\theoremstyle{remark}
\newtheorem{rem}{Remark}[section]
\numberwithin{equation}{section}
\def\ind{{\rm 1\hspace{-0.90ex}1}}
\begin{document}

\title{On  the density of nonlinear statistics}
%{Non-Gaussian functionals: Covariance formula and its applications}
%{Entropic central limit theorem using Stein's kernel}
\author{Nguyen Tien Dung\footnote{Department of Mathematics, VNU University of Science, Vietnam National University, Hanoi, 334 Nguyen
Trai, Thanh Xuan, Hanoi, 084 Vietnam. Email: dung@hus.edu.vn, dung\_nguyentien10@yahoo.com}}

\date{\today}          % Ngay
%\address{Department of Mathematics, FPT University\\ Hoa Lac High Tech Park, Hanoi, Vietnam}
%\email{dung\_nguyentien10@yahoo.com, dungnt@fpt.edu.vn}
\maketitle
\begin{abstract}
In this note, we revisit a classical problem related to the density of nonlinear statistics. We obtain a new representation of densities and, for the first time, a necessary and sufficient condition for the existence of densities  is provided.
\end{abstract}
\noindent\emph{Keywords:} Probability density, nonlinear statistic.\\
{\em 2010 Mathematics Subject Classification:} 60E05, 62E15.
%\section{Introduction}       % Muc dau tien
%{\large Densities, Tail probabilities and }some interval $[a_k,b_k]\subseteq \mathbb{R}$
\section{Introduction}
Let $X=(X_1,X_2,...,X_n)$ be a vector of independent random variables, where each $X_k$ has a density $p_k.$  We are interested in the probability density of nonlinear statistics of the form
$$T:=T(X)=T(X_1,\cdots,X_n),$$
where $T:\mathbb{R}^n\to \mathbb{R}$ is an absolutely continuous function with respect to its each variable.

The ''traditional'' method, oft quoted in school textbooks, to find the density $p_T$ of $T$ is as follows, see e.g. Section 2.5 in \cite{Ross2007}. Assume that there exist the auxiliary functions $T_k=T_k(X_1,\cdots,X_n),2\leq k\leq n$ such that

\noindent$(i)$ the functions $T,T_2,\cdots,T_n$ have continuous partial derivatives and that the Jacobian determinant $J(x_1,\cdots,x_n)\neq 0$ at all
points $(x_1,\cdots,x_n),$ where
$$J(x_1,\cdots,x_n)=\left|
                      \begin{array}{cccc}
                        \frac{\partial T}{\partial x_1} & \frac{\partial T}{\partial x_2} & \cdots & \frac{\partial T}{\partial x_n} \\
                        &&&\\
                        \frac{\partial T_2}{\partial x_1} &\frac{\partial T_2}{\partial x_2} & \cdots & \frac{\partial T_2}{\partial x_n} \\
                        \cdots&&&\\
                        \frac{\partial T_n}{\partial x_1} & \frac{\partial T_n}{\partial x_2} &  \cdots & \frac{\partial T_n}{\partial x_n} \\
                      \end{array}
                    \right|.
$$
\noindent$(ii)$  the equations $y_1=T(x_1,\cdots,x_n),y_2=T_2(x_1,\cdots,x_n),\cdots,y_n=T_n(x_1,\cdots,x_n)$ have a unique solution, say, $x_1=h(y_1,\cdots,y_n),x_2=h_2(y_1,\cdots,y_n),\cdots,x_n=h_n(y_1,\cdots,y_n).$

Under the above assumptions, the joint density function of $(T,T_2,\cdots,T_n)$ is given by
$$p_{T,T_2,\cdots,T_n}(y_1,\cdots,y_n)=\frac{p_1(x_1)\cdots p_n(x_n)}{|J(x_1,\cdots,x_n)|},$$
and hence, we obtain the following expression for the density $p_T$
$$p_T(y_1)=\int_{\mathbb{R}^{n-1}}p_{T,T_2,\cdots,T_n}(y_1,\cdots,y_n)dy_2\cdots dy_n.$$
It can be seen that the traditional method is a quantitative method. It does not gives a clear answer about the existence and  non-existence of densities. In the present paper, we derive a new expression for $p_T.$ Particularly, we obtain a necessary and sufficient condition for the existence of the density of $T.$ Our main tool is a covariance formula established by Cuadras, Theorem 1 in his paper \cite{Cuadras2002} can be restated as follows.

\begin{lem}\label{kflo} Let $X$ be a random variable such that its range is the interval $[a,b]\subseteq \mathbb{R}.$ If $\alpha,\beta$ are two functions defined on $[a,b]$ such that

$(i)$ both functions are of bounded variation,

$(ii)$ $E|\alpha(X)\beta(X)|,E|\alpha(X)|,E|\beta(X)|<\infty,$

\noindent then
$${\rm Cov}(\alpha(X),\beta(X))=\int_a^b\int_a^b (F(x\wedge y)-F(x)F(y))d\alpha(x)d\beta(y),$$
where $F$ is the cumulative distribution function of $X.$ %Theorem \ref{opf}
\end{lem}
Based on Lemma \ref{kflo}, we construct a function $\theta(T)$ to control the existence of densities. We show that the density $p_T$ of $T$  exists if and only if $\theta(T)>0\,\,a.s.$ In addition, $p_T$ can be represented in term of $\theta.$ The rest of the paper is organized as follows. The main results of the paper are formulated and proved in Section \ref{kklf}. Some examples with detailed computations are provided in Section \ref{kklfg}.

%This result will be formulated and proved in the next section.

% Furthermore, when all computations make sense, the density of $T$ can be computed explicitly.

%

\section{The main results}\label{kklf}
In the whole this section, we assume that  the  distribution  of  each  random  variable $X_k(k=1,...,n)$ with distribution function $F_k(x),x\in \mathbb{R},$ is supported on some connected interval $[a_k,b_k]$ of $\mathbb{R}$ and has an a.e. positive density $p_k(x)$ on that interval. Furthermore, we assume that $T:\mathbb{R}^n\to \mathbb{R}$ is an absolutely continuous function with respect to its each variable and $E|T|^2<\infty.$ We also observe that $p_T(x)=p_{T-E[T]}(x-E[T]).$ Hence, for the simplicity, we may and will assume that $E[T]=0.$

%\begin{defn} Let $h:\mathbb{R}^n\to \mathbb{R}$ be an absolutely continuous function with respect to its $k$th variable. We define
%$$\mathcal{L}_kh=\frac{\int_{X_k}^\infty h(X,X_k=x)p_k(x)dx}{p_k(X_k)}$$
%$\mathcal{L}_kh\equiv0$ if the density of $X_k$ does not exist.
%$$\mathcal{L}_kh(x)=\frac{\int_{-\infty}^\infty (F_k(x_k\wedge y)-F_k(x_k)F_k(y))\partial_kh(x,x_k=y)dy}{p_k(x_k)}$$
%\end{defn}
\begin{defn}\label{hkd}The set $h:=\{h_1,\cdots,h_m\},m\geq 1$ is called a decomposition of $T$ if

$(i)$ the function $h_k:\mathbb{R}^n\to \mathbb{R},1\leq k\leq m$ is absolutely continuous with respect to $x_k,$

$(ii)$  $T-E[T]=h_1(X)+\cdots+h_m(X),$

$(iii)$ $E_k[h_k(X)]=0$ for all $1\leq k\leq m$ and $E|h_1(X)|^2+\cdots+E|h_m(X)|^2<\infty,$ where $E_k$ denotes the expectation with respect to $X_k.$
\end{defn}
It should be noted that the decomposition of $T$ always exists. Indeed, we have the following martingale decomposition
$$T-E[T]=\sum\limits_{k=1}^n\left(E[T|X_1,...,X_k]-E[T|X_1,...,X_{k-1}]\right).$$
It is easy to check that the functions $h_k(X):=E[T|X_1,...,X_k]-E[T|X_1,...,X_{k-1}],1\leq k\leq m=n$ satisfy the conditions of Definition \ref{hkd}.

%\begin{defn}Let $h:=\{h_1,\cdots,h_m\}$ is a decomposition of $T.$ We define
%$$\Theta_{T,h}:=\Theta_{T,h}(X)=\sum\limits_{k=1}^m\partial_k T(X)\frac{\int_{a_k}^{b_k} (F_k(X_k\wedge y)-F_k(X_k)F_k(y))\partial_kh(X,X_k=y)dy}{p_k(X_k)}.$$
%When $F=G,$ we write  $\Theta_{F}$ instead of $\Theta_{F,G}.$=\sum\limits_{k=1}^m\partial_k T\mathcal{L}_kh_k
%\end{defn}

The next Proposition plays a key role in our work.
\begin{prop}\label{jfk2} Let $h:=\{h_1,\cdots,h_m\}$ be a decomposition of $T.$
We define
$$\Theta_{T,h}:=\Theta_{T,h}(X)=\sum\limits_{k=1}^m\partial_k T(X)\frac{\int_{a_k}^{b_k} (F_k(X_k\wedge y)-F_k(X_k)F_k(y))\partial_kh(X,X_k=y)dy}{p_k(X_k)},$$
where $\partial_kh:=\frac{\partial_kh}{\partial x_k}$ and $h(X,X_k=y):=h(X_1,...,X_{k-1},y,X_{k+1},...,X_n).$ Then, for any differentiable function $g:\mathbb{R}\to \mathbb{R}$ with bounded derivative, we have
\begin{equation}\label{i6kkdl2}
E[g(T)T]=E[g'(T)\Theta_{T,h}].
\end{equation}
Moreover, the function $\theta(T):=E[\Theta_{T,h}|T]$ does not depend on the choice of decompositions and $\theta(T)\geq 0\,\,\,a.s.$
\end{prop}
\begin{proof}We apply Lemma \ref{kflo} to $\alpha(x)=g(T(X,X_k=x))$ and $\beta(x)=h_k(X,X_k=x)$ and we obtain
\begin{align*}
E_k[g(T)h_k]&=\int_{a_k}^{b_k}\int_{a_k}^{b_k} (F_k(z\wedge y)-F_k(z)F_k(y))g'(T(X,X_k=z))\partial_kT(X,X_k=z)\partial_kh_k(X,X_k=y)dzdy\\
&=\int_{a_k}^{b_k}g'(T(X,X_k=z))\partial_kT(X,X_k=z)\mathcal{L}_kh_k(X,X_k=z)p_k(z)dz\\
&=E_k[g'(T)\partial_kT\mathcal{L}_kh_k],\,\,1\leq k\leq m.
\end{align*}
Hence, by the independence of $X_k's,$ we deduce
$$E[g(T)h_k]=E[g'(T)\partial_kT\mathcal{L}_kh_k],\,\,1\leq k\leq m.$$
As a consequence,
\begin{align*}
E[g(T)T]&=\sum\limits_{k=1}^m E[g(T)h_k]\\
&=E\left[g'(T)\sum\limits_{k=1}^m \partial_kT\mathcal{L}_kh_k\right]\\
&=E[g'(T)\Theta_{T,h}].
\end{align*}
So the formula (\ref{i6kkdl2}) is proved.

We now let $\bar{h}:=\{\bar{h}_1,\cdots,\bar{h}_{\bar{m}}\}$ be another decomposition of $T.$ For any  continuous function $g$ with compact support we have
$$E[g(T)\Theta_{T,h}]=E[G(T)T]=E[g(T)\Theta_{T,\bar{h}}],$$
where $G$ is an antiderivative of $g.$ So we conclude that $E[\Theta_{T,h}|T]=E[\Theta_{T,\bar{h}}|T]\,\,a.s.$

To check the non-negativity of $\theta(T),$ we let $g$ be a bounded non-negative function and set $G(x)=\int_0^x g(t)dt.$ Since $G$ is non-decreasing and $G(0)=0,$ we have $TG(T)\geq 0.$  Therefore, by the formula (\ref{i6kkdl2}), we get
$$E\left[g(T)\theta(T)\right]= E[g(T)\Theta_{T,h}]=E[TG(T)]\geq 0$$
for any bounded non-negative function $g.$ This implies the desired conclusion.
\end{proof}
We now are in a position to state the main result of this paper.
\begin{thm}\label{opf}Let the function $\theta(T)$ be as in Proposition \ref{jfk2}. The law of $T$ has a density $p_T$  with respect to the Lebesgue measure on $\mathbb{R}$ if and only if $\theta(T)>0\,\,a.s.$ Moreover, we have
\begin{equation}\label{jkf4}
p_T(x)=\frac{c}{\theta(x)}\exp\left(-\int_{0}^x \frac{u}{\theta(u)}du\right),\,\,x\in {\rm Supp}(p_T),
\end{equation}
where $c>0$ is a normalized constant.
\end{thm}
\begin{proof} The proof is broken up into three parts.

\noindent{\it  Necessary condition.} We assume that $T$ has a density $p_T.$ Since $T$ is a continuous function and the densities of random variables $X_k's$ have connected support. Those imply that the support of $p_T$ is also a connected interval, say, $[a,b].$ Moreover, we have $a<0<b$ because $E[T]=0.$

We consider the function $\varphi(x):=\int_x^b yp_T(y)dy,x\in (a,b).$ We have $\varphi(b):=\lim\limits_{x\to b}\varphi(x)=0$ because $E|T|<\infty$ and $\varphi(a):=\lim\limits_{x\to a}\varphi(x)=E[T]=0.$ In addition, we have $\varphi(x)>0$ for all $0\leq x<b$ and $\varphi(x)=\varphi(x)-\varphi(a)=-\int_a^x yp_T(y)dy>0$ for all $a<x<0.$ Thus we have $\varphi(x)>0$ for all $x\in (a,b).$

Let $g$ be a continuous function with compact support and set
$G(x)=\int_0^x g(t)dt.$ Then, by integrating by parts, we obtain
\begin{align}
E[G(T)T]&=\int_{-\infty}^\infty G(x)xp_T(x)dx\notag\\
&=-\int_{-\infty}^\infty G(x)d\varphi(x)\notag\\
&=\int_{-\infty}^\infty g(x)\varphi(x)dx\notag\\
&=E\left[g(T)\frac{\varphi(T)}{p_T(T)}\right]\label{ui2}
\end{align}
On the other hand, it follows from the formula (\ref{i6kkdl2}) that
\begin{equation}\label{ui1}
E[G(T)T]=E[g(T)\Theta_{T,h}]=E[g(T)\theta(T)].
\end{equation}
Comparing (\ref{ui2}) and (\ref{ui1}) yields
$$E[g(T)\theta(T)]=E\left[g(T)\frac{\varphi(T)}{p_T(T)}\right]$$
for any continuous function $g$ with compact support.  This implies that
\begin{equation}\label{7ud}
\theta(T)=\frac{\varphi(T)}{p_T(T)}>0\,\,\,a.s.
\end{equation}

\noindent{\it  Sufficient condition.} By the formula (\ref{i6kkdl2})  and a standard approximation argument we have
\begin{equation}\label{eq:2}
E\left[T\int_{-\infty}^{T}\ind_{B}(x)dx\right]=E[\theta(T) \ind_{B}(T)]
\end{equation}
for any Borel set $B\in\mathcal{B}(\mathbb{R})$. Suppose that the Lebesgue measure of $B$ is zero. Then by
\eqref{eq:2} we have $E[\theta(T) \ind_{B}(T)]=0,$ and hence, $P(T\in B)=0$ if $\theta(T)>0\,\,a.s.$ In other words, the condition $\theta(T)>0\,\,a.s.$ implies that
the law of $T$ is absolutely continuous with respect to the Lebesgue measure.

\noindent{\it  Representation formula.}  By the definition of $\varphi,$ it follows from the relation (\ref{7ud}) that
 $$\varphi(x)=\int_x^b \frac{y}{\theta(y)}\varphi(y)dy,\,\,x\in (a,b),$$
The above equation is a linear integral equation and its solution is given by
 $$\varphi(x)=\varphi(0)\exp\left(-\int_{0}^x \frac{u}{\theta(u)}du\right),\,\,x\in (a,b).$$
 Consequently, we deduce
$$p_T(x)=\frac{\varphi(x)}{\theta(x)}=\frac{\varphi(0)}{\theta(x)}\exp\left(-\int_{0}^x \frac{u}{\theta(u)}du\right),\,\,x\in (a,b).$$
The proof of Theorem is complete.
\end{proof}
Let us end this section with some remarks.
%Theorem \ref{opf} still holds true if $T$ is an absolutely continuous function with respect to its each variable.
\begin{rem} If $X_1$ be a discrete random variable with the support $\{e_1,...,e_N\},$ we have
$$P(T\leq x)=\sum\limits_{i=1}^N P(T(e_i,X_2,...,X_n)\leq x)P(X_1=e_i).$$
Hence, we can investigate the density of $T$ by applying Theorem \ref{opf} to $T(e_i,X_2,...,X_n),1\leq i\leq N.$
\end{rem}
\begin{rem} Comparing with the traditional method mentioned in Introduction, the density formula (\ref{jkf4}) gives us another method to study the densities. In general, for complicated statistics $T,$ it is almost impossible to find the exact density function of $T.$ Here we can use the density formula (\ref{jkf4}) to bound $p_T$ as follows. We assume that $\theta(T)$ can be estimated by some positive functions $\theta_1(T)\leq \theta(T)\leq \theta_2(T),$ then
%This is not surprising because the same thing also holds for the density $p_T.$
 $$ \frac{c}{\theta_2(x)}\exp\left(-\int_{0}^x \frac{u}{\theta_1(u)}du\right)\leq p_T(x)\leq \frac{c}{\theta_1(x)}\exp\left(-\int_{0}^x \frac{u}{\theta_2(u)}du\right),\,\,x\in {\rm Supp}(p_T).$$
 \end{rem}
\begin{rem} The relation (\ref{7ud}) itself may be of independent interest. When $p_T(x)$ can be computed by the traditional method, we can use it to compute the conditional expectations by
$$E[\Theta_{T,h}|T=x]=\frac{\int_x^b yp_T(y)dy}{p_T(x)}.$$
See Proposition \ref{jk9} below for an interesting application.
\end{rem}

\section{Examples}\label{kklfg}
We provide here some examples to illustrate the applicability of our results. In Statistical mechanics, the limit distribution of the Curie-Weiss model has the density given by
\begin{equation}\label{ujf2}
p(x)=c_s\exp\left(-\frac{x^{2s}}{2s\,\sigma^2}\right),\,\,x\in \mathbb{R},
\end{equation}
where $s$ is a positive integer number, $\sigma>0$ and $c_s=1/\int_{-\infty}^\infty e^{-\frac{x^{2s}}{2s\,\sigma^2}}dx.$ The reader can consult \cite{Eichelsbacher2010} for more details. Let $X_1,X_2,...,X_n$ be independent random variables with the same density $p(x)$ defined by (\ref{ujf2}). We consider the statistic
$$
W:=\alpha_1 X_1^{2s}+\cdots+\alpha_n X_n^{2s},
$$
where $\alpha_1,\cdots,\alpha_n$ are positive real numbers. We note that, when $s=\sigma=1,$ $X_1,X_2,...,X_n$ are standard normal random variables. In this special case, $W$ has a Chi-square density with $n$ degrees of freedom if $\alpha_1=...=\alpha_n=1$ and the exact distribution of $W$ is unknown if $\alpha_i\neq \alpha_j$ for some $i, j.$ Our next Proposition generalizes this special case to all $s\geq 1.$
\begin{prop} The density $p_W$ of $W$ exists and satisfies the following bounds, for all $x\geq 0,$
\begin{multline}\label{klg}
\frac{c}{2s\,\sigma^2\alpha^\ast\beta^{\frac{\beta}{2s\,\sigma^2\alpha_\ast}}}\exp\left(-\frac{x-\beta}{2s\,\sigma^2\alpha_\ast}\right)
x^{\frac{\beta}{2s\,\sigma^2\alpha_\ast}-1}\leq p_W(x)\\
\leq \frac{c}{2s\,\sigma^2\alpha_\ast\beta^{\frac{\beta}{2s\,\sigma^2\alpha^\ast}}}\exp\left(-\frac{x-\beta}{2s\,\sigma^2\alpha^\ast}\right)
x^{\frac{\beta}{2s\,\sigma^2\alpha^\ast}-1},
\end{multline}
where $c$ is some positive constant, $\alpha_\ast:=\min\limits_{1\leq k\leq n}\alpha_k,$ $\alpha^\ast:=\max\limits_{1\leq k\leq n}\alpha_k$ and $\beta:=E[W]=\sigma^2\sum\limits_{k=1}^n \alpha_k$. Particularly, when $\alpha_1=...=\alpha_n=\frac{1}{s\,\sigma^2},$  $W$ has a Chi-square density with $\frac{n}{s}$ degrees of freedom.
\end{prop}
\begin{proof}It is easy to see that $E[X_k^{2s}]=\sigma^2$ for all $1\leq k\leq n.$ We consider the random variable $T:=W-E[W]=W-\beta.$ By using the decomposition $h_k(x)=\alpha_k(x^{2s}-\sigma^2),\,\,1\leq k\leq n,$ we have
$$\Theta_{T,h}=\sum\limits_{k=1}^n 2s\,\alpha_kX_k^{2s-1}\frac{\int_{-\infty}^{\infty} (F(X_k\wedge y)-F(X_k)F(y))h_k'(y)dy}{p(X_k)},$$
where $F(x)=\int_{-\infty}^xp(y)dy.$ Integration by parts gives
\begin{align*}
\frac{\int_{-\infty}^{\infty} (F(X_k\wedge y)-F(X_k)F(y))h_k'(y)dy}{p(X_k)}&=\frac{\int_{X_k}^\infty h_k(y)p(y)dy}{p(X_k)}\\
&=\frac{\alpha_k\int_{X_k}^\infty y^{2s}p(y)dy-\alpha_k\sigma^2\int_{X_k}^\infty p(y)dy}{p(X_k)}\\
&=\frac{-\alpha_k\sigma^2\int_{X_k}^\infty ydp(y)-\alpha_k\sigma^2\int_{X_k}^\infty p(y)dy}{p(X_k)}\\
&=\alpha_k\sigma^2X_k+\frac{\alpha_k\sigma^2\int_{X_k}^\infty p(y)dy-\alpha_k\sigma^2\int_{X_k}^\infty p(y)dy}{p(X_k)}\\
&=\alpha_k\sigma^2X_k,\,\,1\leq k\leq n.
\end{align*}
We  obtain
$$\Theta_{T,h}=2s\,\sigma^2\sum\limits_{k=1}^n \alpha_k^2X_k^{2s},$$
and hence,
$$2s\,\sigma^2\alpha_\ast(T+\beta)\leq \theta(T)=E[\Theta_{T,h}|T]\leq 2s\,\sigma^2\alpha^\ast(T+\beta)\,\,a.s.$$
Recalling the density formula (\ref{jkf4}), we deduce
\begin{align}
p_T(x)&\leq \frac{c}{2s\,\sigma^2\alpha_\ast(x+\beta)}\exp\left(-\int_{0}^x \frac{u}{2s\,\sigma^2\alpha^\ast(u+\beta)}du\right)\notag\\
&= \frac{c}{2s\,\sigma^2\alpha_\ast(x+\beta)}\exp\left(-\frac{x}{2s\,\sigma^2\alpha^\ast}\right)\left(\frac{x+\beta}{\beta}\right)^{\frac{\beta}{2s\,\sigma^2\alpha^\ast}}\notag\\
&= \frac{c}{2s\,\sigma^2\alpha_\ast\beta^{\frac{\beta}{2s\,\sigma^2\alpha^\ast}}}\exp\left(-\frac{x}{2s\,\sigma^2\alpha^\ast}\right)
\left(x+\beta\right)^{\frac{\beta}{2s\,\sigma^2\alpha^\ast}-1}.\label{klg1}
\end{align}
Similarly,
\begin{align}
p_T(x)&\geq \frac{c}{2s\,\sigma^2\alpha^\ast(x+\beta)}\exp\left(-\int_{0}^x \frac{u}{2s\,\sigma^2\alpha_\ast(u+\beta)}du\right)\notag\\
&=
\frac{c}{2s\,\sigma^2\alpha^\ast\beta^{\frac{\beta}{2s\,\sigma^2\alpha_\ast}}}\exp\left(-\frac{x}{2s\,\sigma^2\alpha_\ast}\right)
\left(x+\beta\right)^{\frac{\beta}{2s\,\sigma^2\alpha_\ast}-1}.\label{klg2}
\end{align}
We now observe that $p_W(x)=p_T(x-\beta).$ So (\ref{klg}) follows from (\ref{klg1}) and (\ref{klg2}).

When $\alpha_1=...=\alpha_n=\frac{1}{s\,\sigma^2},$ we have
$$p_W(x)=\frac{c}{2\beta^{\frac{\beta}{2}}}\exp\left(-\frac{x-\beta}{2}\right)
x^{\frac{\beta}{2}-1},$$
where $\beta=\frac{n}{s}.$ This completes the proof of Proposition.
\end{proof}
\begin{prop}Suppose that $X_1,\cdots,X_n$ are independent standard normal random variables. Let $T:\mathbb{R}^n\to \mathbb{R}$ be a continuously differentiable function such that
$$\alpha_k\leq \partial_k T(x)\leq \beta_k,\,\,\forall\,x\in \mathbb{R},$$
where $\alpha_k\,\beta_k,k=1,...,n$ are non-negative real numbers. In addition, we assume $T=T(X)$ is a centered random variable with finite variance. Then, the density of $T$ exists and satisfies
$$
\frac{c}{\sigma_2^2}\exp\left({-\frac{x^2}{2\sigma_1^2}}\right)\leq p_T(x)\leq \frac{c}{\sigma_1^2}\exp\left({-\frac{x^2}{2\sigma_2^2}}\right),\,\,x\in {\rm Supp}(p_T),
$$
where $c$ is some positive constant,  $\sigma_1^2:=\sum\limits_{k=1}^n \alpha^2_k$ and $\sigma_2^2:=\sum\limits_{k=1}^n \beta^2_k.$ The preceding bounds are sharp because we have equalities when $T$ is a linear combination of $X_1,\cdots,X_n.$
\end{prop}
\begin{proof} We choose to use the following decomposition of $T:$ $$h_k(X):=E[T|X_1,...,X_k]-E[T|X_1,...,X_{k-1}],1\leq k\leq n.$$  We have $\partial_k h_k(X)=E[\partial_kT|X_1,...,X_k]$ and hence,
$$\alpha_k\leq \partial_k h_k(X)\leq \beta_k\,\,a.s.$$
for all $1\leq k\leq n.$ Recalling the definition of $\Theta_{T,h},$ we deduce
$$
\sum\limits_{k=1}^n \alpha_k^2\frac{\int_{-\infty}^{\infty} (F(X_k\wedge y)-F(X_k)F(y))dy}{p(X_k)}\leq \Theta_{T,h}\leq \sum\limits_{k=1}^n \beta_k^2\frac{\int_{-\infty}^{\infty} (F(X_k\wedge y)-F(X_k)F(y))dy}{p(X_k)},
$$
where $p,F$ denote the density and cumulative distribution function of standard normal random variable, respectively. By straightforward computations, we get
$$\frac{\int_{-\infty}^{\infty} (F(X_k\wedge y)-F(X_k)F(y))dy}{p(X_k)}=1,\,\,\,1\leq k\leq n.$$
So it holds that
$$
\sum\limits_{k=1}^n \alpha_k^2\leq \Theta_{T,h}\leq \sum\limits_{k=1}^n \beta_k^2\,\,a.s.
$$
Thus we have $\sigma_1^2\leq \theta(T)\leq \sigma_2^2\,\,a.s.$ and by using the density formula (\ref{jkf4}) we obtain
$$\frac{c}{\sigma_2^2}\exp\left({-\frac{x^2}{2\sigma_1^2}}\right)\leq p_T(x)\leq \frac{c}{\sigma_1^2}\exp\left({-\frac{x^2}{2\sigma_2^2}}\right),\,\,x\in {\rm Supp}(p_T).$$
%To see the fact ${\rm Supp}(p_T)=\mathbb{R},$ we recall that ${\rm Supp}(p_T)=[a,b]$ and the function $\varphi(x)$ defined in the proof of Theorem \ref{opf} satisfies $\varphi(a)=\varphi(b)=0.$
%\begin{align*}
%\varphi(x)=\int_x^b yp_T(y)dy\geq \int_x^b y\frac{c}{\sigma_2^2}\exp\left({-\frac{y^2}{2\sigma_1^2}}\right)dy,x\in (a,b)
%\end{align*}
The proof of Proposition is complete.
\end{proof}
\begin{prop}\label{jk9} Suppose that $X_1,\cdots,X_n$ are uniformly distributed on the interval $[0,1].$ We have the following identity
\begin{equation}\label{uif}
E[X_1^2+\cdots+X_n^2|X_1+\cdots+X_n=x]=x-2\frac{\int_{x}^{n} (y-n/2)p_Z(y)dy}{p_Z(x)},\,\,0\leq x\leq n,
\end{equation}
where $p_Z(x)=0$ if $x\notin [0,n]$ and
$$p_Z(x)=\frac{1}{(n-1)!}\sum\limits_{0\leq k\leq x}(-1)^k\binom{n}{k}(x-k)^{n-1},\,\,0\leq x\leq n.$$
\end{prop}
\begin{proof}It is well know that $p_Z(x)$ is the density of $Z:=X_1+\cdots+X_n.$ Then, the density of $T:=Z-E[Z]=Z-\frac{n}{2}$ is given by
$$p_T(x)=p_Z(x+n/2)=\frac{1}{(n-1)!}\sum\limits_{0\leq k\leq x+ \frac{n}{2}}(-1)^k\binom{n}{k}(x+ \frac{n}{2}-k)^{n-1},\,\, -\frac{n}{2}\leq x\leq \frac{n}{2}.$$
We use the functions $h_k(x)=x-\frac{1}{2},\,1\leq k\leq n$ as a decomposition of $T.$ We have
$$\Theta_{T,h}=\sum\limits_{k=1}^n\int_{0}^{1} (F(X_k\wedge y)-F(X_k)F(y))dy=\frac{1}{2}\sum\limits_{k=1}^n(X_k-X_k^2),$$
where $F$ denotes the cumulative distribution function of the uniform distribution on $[0,1].$ As a consequence, we get
\begin{align*}
E[X_1^2+\cdots+X_n^2|X_1+\cdots+X_n=x]&=E[X_1+\cdots+X_n-2\Theta_{T,h}|X_1+\cdots+X_n=x]\\
&=x-2E[\Theta_{T,h}|T=x-n/2]\\
&=x-2\frac{\int_{x-n/2}^{n/2} yp_T(y)dy}{p_T(x-n/2)}\\
&=x-2\frac{\int_{x}^{n} (y-n/2)p_Z(y)dy}{p_Z(x)}.
\end{align*}
This finishes the proof of Proposition. To the best of our knowledge, the identity (\ref{uif}) is new and it is difficult to be proven directly.
%in this way we will obtain some interesting (and somehow unexpected) identities for functions of the Brownian motion, which are difficult to be proven directly. That is
\end{proof}

\end{document}